\documentclass[12pt]{amsproc}
\usepackage{amssymb}
\usepackage{stmaryrd}
\usepackage{enumerate}
\usepackage[pagebackref,colorlinks,linkcolor=blue,citecolor=blue,urlcolor=blue]{hyperref}
\usepackage{amsrefs}

\theoremstyle{plain}
\newtheorem{theorem}{Theorem}
\newtheorem{corollary}[theorem]{Corollary}

\newtheorem{lemma}[theorem]{Lemma}

\theoremstyle{definition}
\newtheorem{example}[theorem]{Example}

\renewcommand{\geq}{\geqslant}
\renewcommand{\leq}{\leqslant}
\renewcommand{\ge}{\geqslant}
\renewcommand{\le}{\leqslant}

\newcommand{\lhdeq}{\trianglelefteqslant}    
\newcommand{\rhdeq}{\trianglerighteqslant}   

\newcommand{\C}{{\mathcal C}}
\newcommand{\eps}{\varepsilon}
\newcommand{\F}{\mathbb{F}}
\newcommand{\Z}{\mathsf{Z}}

\newcommand{\alt}{\mathrm{Alt}}
\newcommand{\sym}{\mathrm{Sym}}
\newcommand{\CC}{\mathrm{C}}
\newcommand{\Aut}{\mathrm{Aut}}
\newcommand{\Inn}{\mathrm{Inn}}
\newcommand\GammaL{\Gamma\mathrm{L}}
\newcommand{\GL}{\mathrm{GL}}
\newcommand{\GO}{\mathrm{GO}}
\newcommand{\Out}{\mathrm{Out}}
\newcommand{\PGL}{\mathrm{PGL}}
\newcommand{\PSL}{\mathrm{PSL}}
\newcommand{\PSU}{\mathrm{PSU}}
\newcommand{\SL}{\mathrm{SL}}
\newcommand{\Sp}{{\mathrm{Sp}}}
\newcommand{\GSp}{{\mathrm{GSp}}}
\newcommand{\GU}{\mathrm{GU}}
\newcommand{\SU}{\mathrm{SU}}

\begin{document}


\title[$\textup{\sc composition factors of order $p$}$]{The number of composition factors of order $p$ in completely reducible groups of characteristic $p$}

\author[M.~Giudici, S.~P.~Glasby, C.~H.~Li, G.~Verret]{Michael Giudici, S.~P.~Glasby, Cai Heng Li, Gabriel Verret} 

\address{Michael Giudici, S.~P.~Glasby$^*$ and Gabriel Verret$^\dag$,
  \newline\indent Centre for the Mathematics of Symmetry and Computation, 
  \newline\indent The University of Western Australia, 
  \newline\indent 35 Stirling Highway, Crawley, WA 6009, Australia.} 
  
\address{Cai Heng Li, Department of Mathematics, 
  South University of Science and
  \newline\indent Technology of China,
  Shenzhen, Guangdong 518055, P. R. China.}
\
\address{$^*$Also affiliated with The Department of Mathematics, 
  \newline\indent University of Canberra, ACT 2601, Australia.}
  
\address{$\dag$Also affiliated with FAMNIT, University of Primorska, 
\newline\indent Glagolja\v{s}ka 8, SI-6000 Koper, Slovenia.
\newline\indent Current address: Department of Mathematics, The University of Auckland,
\newline\indent Private Bag 92019, Auckland 1142, New Zealand.}

\email{Michael.Giudici@uwa.edu.au; URL: www.maths.uwa.edu.au/$\sim$giudici/}
\email{Stephen.Glasby@uwa.edu.au; URL: www.maths.uwa.edu.au/$\sim$glasby/}
\email{lich@sustc.edu.cn; URL: www.sustc.edu.cn/en/math\_faculty/f/CAIHENGLI}
\email{g.verret@auckland.ac.nz}

\thanks{This research was supported by the Australian Research Council grants DE130101001 and DP150101066.}




\begin{abstract}
Let $q$ be a power of a prime $p$ and let $G$ be a completely reducible subgroup of $\GL(d,q)$. We prove that the number of composition factors of $G$ that have prime order~$p$ is at most $(\eps_q d-1)/(p-1)$, where $\eps_q$ is a function of $q$ satisfying $1\le\eps_q\le3/2$. For every $q$, we give examples showing this bound is sharp infinitely often. 
\end{abstract}

\maketitle

\section{Introduction}\label{S1}

All groups considered in this paper are finite. Given a group $G$ and a prime $p$, let $c_p(G)$ denote the number of composition factors of $G$ of order~$p$. Our main theorem is the following.

\begin{theorem}\label{T1}
Let $q$ be a power of a prime $p$, say $q=p^f$. If $G$ is a completely reducible subgroup of $\GL(d,q)$ with $r$ irreducible components, then
\begin{equation}\label{E1}
   c_p(G)\le\frac{\eps_qd-r}{p-1},
    \quad\textup{where }\eps_q= 
\begin{cases}
    \frac{4}{3}&\textup{if $p=2$ and $f$ is even,}\\
    \frac{p}{p-1}&\textup{if $p$ is a Fermat prime,}\\
    1&\textup{otherwise.}
\end{cases}
\end{equation}
\end{theorem}

Recall that a \emph{Fermat prime} is a prime of the form $2^{2^n}+1$ for some $n\ge0$, and that a subgroup $G$ of $\GL(V)$ is called \emph{completely reducible} if $V$ is a direct sum of irreducible $G$-modules.

Our motivation for Theorem~\ref{T1} arose from studying transitive permutation groups admitting paired orbitals with non-isomorphic subconstituents. In the case when both subconstituents are quasiprimitive, Knapp proved that one must be an epimorphic image of the other~\cite[Theorem~3.3]{Knapp}. This naturally led us to investigate the question of when a quasiprimitive group can be a non-trivial epimorphic image of another quasiprimitive group of the same degree. In an upcoming paper~\cite{GGLV}, we show that this is very rare. Our proof relies on Theorem~\ref{T1} in the case when both quasiprimitive groups are of affine type. 

Let $G$ be a group and let  $a(G)$ be the product of the orders of the abelian composition factors of $G$. Note that $c_p(G)\le\log_p(a(G))$ so upper bounds on $a(G)$ yield upper bounds on $c_p(G)$. It is proved in~\cite[Theorem~6.5]{GMP} that, if $G$ is a completely reducible subgroup of $\GL(d,p^f)$,  then $a(G)\le \beta^{-1}(p^{fd})^\gamma$ where $\beta=24^{1/3}$ and $\gamma=\log_9(48\beta)$, and thus $c_p(G)\le \gamma df-\log_p\beta< \gamma df$.  Our bound improves on this because it is independent of~$f$, it
involves the denominator $p-1$, and  $\eps_q<\gamma\approx 2.244$. Similarly, if  $G$ is a primitive group of degree $n$  and $p\mid n$, then it follows from~\cite[Corollary~6.7]{GMP} that $a(G)\le\beta^{-1}n^{\gamma+1}$ and hence $c_p(G)\le(\gamma+1)\log_p(n)$,  whereas our result implies $c_p(G)\le d+\frac{\eps_pd-1}{p-1}$ if $G$ is primitive of affine type and degree $n = p^d$.

After some preliminary results in Section~\ref{S2}, we exhibit some examples in Section~\ref{sec:Examples} which show that, for every prime power $q$, Theorem~\ref{T1} is sharp infinitely often. In particular, $\eps_q$ is best possible. The bound in Theorem~\ref{T1} can be sharpened (if $q$ is not an odd power of 2) to $\eps_p(G)\le\frac{\eps_q d-s}{p-1}$ where~$s$ is the number of \emph{absolutely} irreducible components of $G$ since $G$ remains completely reducible over the algebraic closure of $\F_q$ by~\cite[\S VII.2]{HB2}.

The proof of Theorem~\ref{T1} is given in Section~\ref{S3}. The main idea is to use induction on~$d$ and then split into cases, according to Aschbacher's classification of the subgroups of $\GL(d,p^f)$. The hardest case is when $G$ is a projectively almost simple absolutely irreducible `$\C_9$ group' with a `non-geometric' linear action. We conclude with Corollary~\ref{Cor2}, which bounds $c_p(G/O_p(G))$ for $G$ an arbitrary subgroup of $\GL(d,p^f)$.

\section{Preliminaries}\label{S2}
Throughout the paper, $p$ will always denote a prime. Given a positive integer $n$, let $n_p$ denote the highest power of $p$ that divides $n$ and let $\CC_n$ denote a cyclic group of order~$n$. By Clifford's theorem~\cite{Cliff}, a normal subgroup of a completely reducible group is also completely reducible. This fact will be used repeatedly. The following lemmas will also be used repeatedly, sometimes without comment. 

\begin{lemma}\label{Lemma:new}\label{lem:n!}
If $r$ is a positive integer, then $\log_p r_p\le \log_p(r!)_p\leqslant (r-1)/(p-1)$.
\end{lemma}
\begin{proof}
The first inequality is obvious. Consider the $p$-adic expansion $r=\sum_{k\ge0} d_kp^k$ of $r$, with `digits' $d_k\in\{0,1,\dots,p-1\}$ for each $k\ge0$. Legendre proved  that $\log_p(r!)_p=\sum_{k\ge1}\lfloor r/p^k\rfloor=(r-s_p(r))/(p-1)$, where $s_p(r)=\sum_{k\ge0} d_k$. The second inequality follows since $s_p(r)\ge1$.
\end{proof}

\begin{lemma}\label{L1} Let $G$ be a group.
\begin{enumerate}[{\rm (a)}]
\item\label{L1a} If $1=G_m\lhdeq G_{m-1}\lhdeq\cdots\lhdeq G_0=G$ is a subnormal series for $G$, then $c_p(G)=\sum_{i=1}^mc_p(G_{i-1}/G_i)$. 
\item\label{L1x} $c_p(G)\leq \log_p|G|_p$. If $G$ is $p$-soluble, then $c_p(G)=\log_p|G|_p$. 
\item\label{L1e}\label{L1d} If $G$ is a subgroup of a group $\Gamma$, then $c_p(G)\le\log_p|\Gamma|_p$. In particular, if $G\le\sym(r)$, then $c_p(G)\le(r-1)/(p-1)$.
\item\label{L1b} If $G$ is a subgroup of a direct product $H_1\times\cdots\times H_r$ where the projection maps $\pi_i\colon G\to H_i$ are surjective for $1\le i\le r$, then $c_p(G)\le c_p(H_1)+\cdots+c_p(H_r)$.
\item\label{L1c} If $G$ is a subgroup of a central product $H_1\circ\cdots\circ H_r$ where the projection maps $G\to H_i$ are surjective for $1\le i\le r$, then $c_p(G)\le c_p(H_1)+\cdots+c_p(H_r)$.
\end{enumerate}
\end{lemma}
\begin{proof} We prove these in order.

\eqref{L1a} The given subnormal series for $G$ can be refined to a composition series for $G$. The result now follows from the definition of $c_p(G)$.

\eqref{L1x} The first claim is obvious. If $G$ is $p$-soluble, then, by definition, each composition factor has order $p$, or coprime to~$p$. The result now follows from the definition of~$c_p(G)$.

\eqref{L1e} Since $|G|_p\le|\Gamma|_p$, we have $c_p(G)\le\log_p|\Gamma|_p$. The second sentence follows from Lemma~\ref{lem:n!}.

\eqref{L1b} Let $G_0=G$. For $1\le i\le r$, let $\pi_i\colon G\to H_i$ be the projection map, let $K_i=\ker(\pi_i)$ and let $G_i=G\cap K_1\cap \cdots \cap K_i$. Note that $G_r=1$. Hence $c_p(G)=\sum_{i=1}^r c_p(G_{i-1}/G_i)$ by~\eqref{L1a}. However,
$$
  \frac{G_{i-1}}{G_i}=\frac{G_{i-1}}{G_{i-1}\cap K_i}
    \cong\frac{G_{i-1}K_i}{K_i}\lhdeq\frac{G}{K_i}\cong H_i.
$$
Thus $c_p(G_{i-1}/G_i)\le c_p(H_i)$ by~\eqref{L1a}, and hence $c_p(G)\le c_p(H_1)+\cdots+c_p(H_r)$.

\eqref{L1c} Let $H=H_1\times\cdots\times H_r$ and let $N$ be a normal subgroup of $H$ such that $H/N= H_1\circ\cdots\circ H_r$. Let $\Gamma$ be the preimage of $G$ in $H$. The projection maps $\Gamma\to H_i$ are surjective, hence $c_p(\Gamma)\le\sum_{i=1}^r c_p(H_i)$ by~\eqref{L1b}. The result follows since $c_p(G)\le c_p(\Gamma)$.
\end{proof}

\section{Examples}\label{sec:Examples}
\begin{lemma}\label{L0}
Let $p$ be a prime, let $r\geq 1$, let $q$ be a prime-power and let $\Gamma_1$ be an irreducible subgroup of $\GL(r,q)$.
For every $n\geq 2$, let $\Gamma_n=\Gamma_{n-1}\wr \CC_p$. Then, for every $n\geq 2$, $\Gamma_n$ is an imprimitive subgroup of $\GL(d_n,q)$ where $d_n=rp^{n-1}$. Furthermore,
\begin{equation*}
  c_p(\Gamma_n)=\frac{\eps d_n-1}{p-1} \quad\textrm{ where} \quad\eps=\frac{c_p(\Gamma_1)(p-1)+1}{r}.
\end{equation*}
\end{lemma}
\begin{proof}
We first prove by induction that $\Gamma_n$ is an irreducible subgroup of $\GL(d_n,q)$. This is true for $n=1$. Assume now that $n\geq 2$ and $\Gamma_{n-1}$ is an irreducible subgroup of $\GL(d_{n-1},q)$. Let $V=(\F_q)^{d_n}$ be the natural
$\Gamma_n$-module. Restricting to the base group $N=(\Gamma_{n-1})^p$
of $\Gamma_n$, $V$ is a direct sum $V_1\oplus\cdots\oplus V_p$ of pairwise
nonisomorphic irreducible $N$-modules each of dimension $d_{n-1}$. Hence $\Gamma_n$ is an irreducible subgroup of $\GL(d_n,q)$ by Clifford's Theorem~\cite{Cliff}. In particular,
$\Gamma_n$ is imprimitive for $n\geq 2$. The formula for $c_p(\Gamma_n)$ is
true when $n=1$ as $d_1=r$ and $c_p(\Gamma_1)=(\eps r-1)/(p-1)$.
By Lemma~\ref{L1}(\ref{L1a}), $c_p(\Gamma_n)=pc_p(\Gamma_{n-1})+1$. Hence the
the formula for $c_p(\Gamma_n)$ also follows by induction.
\end{proof}

Using Lemma~\ref{L0}, we now give three families of examples that show that the bound in Theorem~\ref{T1} is best possible.


\begin{example}\label{Ex3}
Let $q$ be a power of a prime $p$ and let $\Gamma_1=\GammaL(1,p^p)\cong\GL(1,p^p)\rtimes \CC_p$. Note that $\Gamma_1$ is an absolutely irreducible subgroup of $\GL(p,p)$. Consequently, $\Gamma_1$ is an irreducible subgroup of $\GL(p,q)$. Note also that $c_p(\Gamma_1)=1$. Applying Lemma~\ref{L0} with $r=p$ yields, for every $n\geq 1$, an irreducible subgroup $\Gamma_n$ of $\GL(d_n,q)$ with $c_p(\Gamma_n)=(d_n-1)/(p-1)$, where $d_n=p^n$.
\end{example}

\begin{example}\label{Ex6}
Let $q$ be an even power of $2$ and let $\Gamma_1=\GU(3,2)\cong 3^{1+2}\rtimes\SL(2,3)$. Note that $\Gamma_1$ is an absolutely irreducible subgroup of $\GL(3,2^2)$. Thus, $\Gamma_1$ is an irreducible subgroup of $\GL(3,q)$. Note also  that $c_2(\Gamma_1)=3$.  Applying Lemma~\ref{L0} with $(p,r)=(2,3)$ yields, for every $n\geq 1$, an irreducible subgroup $\Gamma_n$ of $\GL(d_n,q)$ with $c_2(\Gamma_n)=(4/3)d_n-1$, where $d_n=3\cdot 2^{n-1}$. 
\end{example}

\begin{example}\label{Ex5}
Let $p=2^m+1$ be a Fermat prime, let $q$ be a power of $p$, let $E$ denote an extraspecial group of order $2^{1+2m}$ and type $-$, let $P$ be a Sylow $p$-subgroup of the orthogonal group $\textrm{GO}^{-}(2m,2)$, and let $\Gamma_1=E\rtimes P$.  Note that $\Gamma_1$ is an absolutely irreducible subgroup of $\GL(2^m,p)=\GL(p-1,p)$. Consequently, $\Gamma_1$ is an irreducible subgroup of $\GL(p-1,q)$. Note also that $|P|=p$ and $c_p(\Gamma_1)=1$. Applying Lemma~\ref{L0} with $r=p-1$ yields, for every $n\geq 1$, an irreducible subgroup $\Gamma_n$ of $\GL(d_n,q)$ with $c_p(\Gamma_n)=(\eps d_n-1)/(p-1)$, where $\eps=p/(p-1)$ and $d_n=(p-1)p^{n-1}$. 
\end{example}

The three examples above together show that, for every prime power $q$, Theorem~\ref{T1} is sharp infinitely often.
In Theorem~\ref{T1} and these examples, the prime $p$ divides the field size. If $p$ does not divide $q$, then $c_p(G)$ cannot be bounded by a function of only $d$ and $p$, as the following example shows.


\begin{example}\label{Ex9}
Let $p\ne2$ and $r$ be primes such that $r\equiv1\pmod p$, let $f$ be a positive power of $p$,  let $q=r^f$ and let $G=\GL(d,q)$. Note that $G/\SL(d,q)$ is cyclic of order $q-1$ hence $c_p(G)\geq (r^f-1)_p=(r-1)_pf_p=(r-1)_pf$.
\end{example}

\section{Proof of Theorem~\ref{T1}}\label{S3}
Let $p$ be a prime, let $f$ be a positive integer and let $q=p^f$. Let $V=(\F_q)^d$, viewed as a vector space over $\F_q$, and let $G$ be a completely reducible subgroup of $\GL(V)\cong\GL(d,q)$. It is also useful to note that $\eps_q\ge 1$.

Our proof now proceeds by induction on pairs $(d,f)$ where we use the lexicographic ordering $(d_1,f_1)<(d_2,f_2)$ if $d_1<d_2$, or $d_1=d_2$ and $f_1<f_2$. The base case when $d=f=1$ is trivial.

Since $\GL(d,q)/\SL(d,q)$ has order $q-1$ and thus coprime to~$p$, it follows by Lemma~\ref{L1}(\ref{L1a}) that $c_p(G)=c_p(G\cap\SL(d,q))$. We henceforth assume that $G\le\SL(d,q)$. Let $Z=\Z(\SL(d,q))$. Note that $Z$ has order $\gcd(d,q-1)$ which is coprime to~$p$  hence $c_p(G)=c_p(GZ)$. We thus assume henceforth that $Z\le G$.

In fact, $c_p(\SL(d,q))=0$ unless $d=2$ and $q\in\{2,3\}$, in which case $c_p(\SL(d,q))=1$. In both cases,~\eqref{E1} holds hence we assume $G<\SL(d,q)$.

Our proof relies heavily on Aschbacher's Theorem characterising the subgroups of $\GL(d,q)$ that do not contain $\SL(d,q)$, which asserts that $G$ lies in at least one of the following nine classes~\cite{A}.

\begin{itemize}
\item[$\mathcal{C}_1$] \textbf{(reducible subgroups):} In this case, $G$ fixes some proper nonzero subspace of~$V$.
\item[$\mathcal{C}_2$] \textbf{(imprimitive subgroups):} In this case,  $G$ fixes some decomposition $V=V_1\oplus\cdots\oplus V_r$, where
  $r\geq 2$ and each $V_i$ has dimension $d/r$. In particular, $G\leqslant \GL(d/r,q) \wr \sym(r)$.
\item[$\mathcal{C}_3$] \textbf{(extension field subgroups):} In this case,  $G$ preserves the structure of $V$ as a $(d/r)$-dimensional vector space over $\F_{q^r}$ for some $r\geq 2$. In this case, $G\leqslant \GL(d/r,q^r)\rtimes \CC_r$.
\item[$\mathcal{C}_4$] \textbf{(tensor product subgroups):} In this case, $G$ preserves a tensor product decomposition $V=U\otimes W$ with $d=\dim(U)\dim(W)$ and $\dim(U)\neq \dim(W)$. In particular, $G\leqslant \GL(U)\circ \GL(W)$.
\item[$\mathcal{C}_5$] \textbf{(subfield subgroups):} In this case,  $q=q_0^r$ for some $r\geq 2$ and $G\leqslant \GL(d,q_0)\cdot\Z(\GL(d,q))$.
\item[$\mathcal{C}_6$] \textbf{(symplectic type $r$-groups):} In this case,  there is a prime $r$ such that $d=r^m$ and an absolutely irreducible normal $r$-subgroup $R$ of $G$ such that $R/\Z(R)$ is elementary abelian of rank~$2m$.
\item[$\mathcal{C}_7$] \textbf{(tensor-imprimitive subgroups):} In this case,  $G$ preserves the tensor product decomposition $V=V_1\otimes \cdots\otimes V_r$, where each  $V_i$ has dimension $n$ and $d=n^r$. In particular, $G\leq (\GL(n,q)\circ \cdots\circ \GL(n,q))\rtimes \sym(r)$.
\item[$\mathcal{C}_8$] \textbf{(classical groups):} In this case,  $G$ preserves a nondegenerate alternating, hermitian or quadratic form on $V$. Moreover, $G$ contains one of $\Sp(d,q)'$, $\SU(d,\sqrt{q})$ or $\Omega^\varepsilon(d,q)$, where $\varepsilon\in\{\pm,\circ\}$. For more details, see~\S\S\ref{C8}.
\item[$\mathcal{C}_9$] \textbf{(nearly simple groups):} In this case,  $G/Z$ is an almost simple group with socle $N/Z$ such that $Z\leq N$ and $N$ is absolutely irreducible.
\end{itemize}

We now consider these classes one by one.

\subsection{$G\in\C_1$}\label{C1}

As $G\in\C_1$ is completely reducible, $G$ preserves a direct sum decomposition $V=V_1\oplus\cdots\oplus V_r$ with $r\geq 2$ where the restriction $G_i$ of $G$ to $V_i$ is irreducible.  By induction, we have $c_p(G_i)\le\frac{\eps_q d_i-1}{p-1}$ where $d_i=\dim(V_i)$ for each $i$. Since $G\le G_1\times\cdots\times G_r$ and $G$ projects onto each $G_i$, Lemma~\ref{L1}(\ref{L1b}) implies
\[
  c_p(G)\le c_p(\prod_{i=1}^rG_i)=\sum_{i=1}^rc_p(G_i)\le\sum_{i=1}^r \frac{\eps_qd_i-1}{p-1}
  =\frac{\eps_q(\sum_{i=1}^rd_i)-r}{p-1}=\frac{\eps_qd-r}{p-1}.
\]

\subsection{$G\in\C_2$}\label{C2}
In this case $G$ is irreducible and preserves a direct sum decomposition $V=V_1\oplus\cdots\oplus V_r$ with $r\geq 2$. Thus $G$ acts transitively on the set $\Omega=\{V_1,\ldots, V_r\}$. Let $N$ be the kernel of the action of $G$ on $\Omega$, and let $N_i$ denote the restriction of $N$ to $V_i$.  The stabiliser $G_i$ in $G$ of the subspace $V_i$ is irreducible on $V_i$. Thus $c_p(G_i)\le\frac{\eps_q d_i-1}{p-1}$ by induction where $d_i=\dim(V_i)=d/r$ for each $i$. By definition, the projection maps $N\to N_i$ are surjective for all $i$ and, moreover, $N\leq N_1\times\cdots\times N_r$. It follows by Lemma~\ref{L1}(\ref{L1b}) that $c_p(N)\le \sum_{i=1}^r c_p(N_i)$, and $N_i\lhdeq G_i$ implies $c_p(N_i)\le c_p(G_i)$.  On the other hand, $G/N$ is isomorphic to a subgroup of $\sym(r)$ hence $c_p(G/N)\le (r-1)/(p-1)$ by Lemma~\ref{L1}(\ref{L1d}). By Lemma~\ref{L1}(\ref{L1a}), we have
\begin{align*}
  c_p(G)&=c_p(G/N)+c_p(N)\leq c_p(G/N)+ \sum_{i=1}^r c_p(N_i)\leq c_p(G/N)+ \sum_{i=1}^r c_p(G_i)\\
   &\le \frac{r-1}{p-1}+\sum_{i=1}^r \frac{\eps_qd_i-1}{p-1}=\frac{r-1}{p-1}+\frac{\eps_qd-r}{p-1}=\frac{\eps_qd-1}{p-1},
\end{align*}
as desired. From now on, we assume that $G$ is irreducible and primitive. Therefore every normal subgroup $N$ of $G$ acts completely reducibly, indeed homogeneously.

\subsection{$G\in\C_3$}\label{C3}
In this case, there exists $r\geq 2$ such that $G\leq\GL(d/r,q^r)\rtimes \CC_r$. Let  $N=G\cap \GL(d/r,q^r)$. By the inductive hypothesis, we have $c_p(N)\leq\displaystyle{\frac{\eps_{q^r}(d/r)-1}{p-1}}$.  Furthermore, by Lemma~\ref{Lemma:new}, 
$$c_p(G/N)\leq c_p(\CC_r)=\log_p r_p\le\frac{r-1}{p-1}.$$
If $d=r$, then $N\leqslant \GL(1,q^d)$ and hence $|N|$ is coprime to $p$ and $c_p(N)=0$. By Lemma~\ref{L1}(\ref{L1a}), we have 
\begin{equation*}
  c_p(G)=c_p(G/N)\le\frac{r-1}{p-1}=\frac{d-1}{p-1}\leq\frac{\eps_qd-1}{p-1}.
\end{equation*}
We may thus assume that $d\geq 2r$. Note that 
\begin{equation*}
  c_p(G)=c_p(N)+c_p(G/N)\le\frac{\eps_{q^r}(d/r)-1}{p-1}+\frac{r-1}{p-1}
  =\frac{\eps_{q^r}(d/r)+r-2}{p-1}.
\end{equation*}
It thus suffices to show
\begin{equation}\label{Bris}
\eps_{q^r}(d/r)+r-1\leq \eps_q d.
\end{equation}
Suppose now that $\eps_{q^r}\leq\eps_{q}$. In this case, it suffices to show $\eps_q d/r+r-1\leq \eps_q d$ which is equivalent to  $r-1\leq \eps_qd(r-1)/r$, and hence equivalent to $r\leq \eps_q d$. Since $\eps_q\geq 1$, the latter holds and so we may thus assume that $\eps_q<\eps_{q^r}$. From the definition of $\eps$, it follows that $\eps_q=1$ and $\eps_{q^r}=4/3$.  Therefore,~\eqref{Bris} becomes

\begin{equation*}
\frac{4d}{3r}+r-1\leq d.
\end{equation*}
This is equivalent to $r^2-r\leq d(r-4/3)$ which holds since $d\geq 2r$ and $r\geq 2$.

\subsection{$G\in\C_4$ or $\C_7$}\label{C47}
In this case, $G$ preserves a non-trivial decomposition of $V$ as a tensor product, say $V=V_1\otimes\cdots\otimes V_r$ where $r\geq 2$. If $G\in\C_4$,  then $r=2$ and $G$ fixes $V_1$ and $V_2$, otherwise $G$ permutes the factors $V_1,\dots,V_r$. Note that $d=\prod_{i=1}^r d_i$, where $d_i=\dim(V_i)\geq 2$. Let $N$ be the kernel of the action of $G$  on the set $\{V_1,\ldots, V_r\}$ and let $N_i$ denote the restriction of $N$ to $V_i$. By definition, the projection maps $N\to N_i$ are surjective for all $i$ and, moreover, $N\leq N_1\circ\cdots\circ N_r$. It follows by Lemma~\ref{L1}(\ref{L1c}) that $c_p(N)\le \sum_{i=1}^r c_p(N_i)$. By Clifford's Theorem, a subnormal subgroup of a completely reducible group is completely reducible. Since $N_i \trianglelefteqslant N \trianglelefteqslant G$, $N_i$ is completely reducible on $V$ and on $V_i$. Hence, by the inductive hypothesis, we have $c_p(N_i)\le(\eps_qd_i-1)/(p-1)$. On the other hand, $G/N$ is isomorphic to a subgroup of $\sym(r)$ hence $c_p(G/N)\le (r-1)/(p-1)$ by Lemma~\ref{L1}(\ref{L1d}). By Lemma~\ref{L1}(\ref{L1a}), we have
\begin{align*}
  c_p(G)&=c_p(G/N)+c_p(N)\leq  c_p(G/N)+\sum_{i=1}^r c_p(N_i)\le \frac{r-1}{p-1}+\sum_{i=1}^r \frac{\eps_qd_i-1}{p-1}\\
   &= \frac{\eps_q(\sum_{i=1}^r d_i)-1}{p-1}\le \frac{\eps_q(\prod_{i=1}^r d_i)-1}{p-1}=\frac{\eps_qd-1}{p-1},
\end{align*}
where the last inequality follows from the fact that $d_i\geq 2$ for all $i$. This completes the proof of this case.

\subsection{$G\in\C_5$}\label{C5}
In this case, $G\leqslant \GL(d,q_0)\Z(\GL(d,q))$ where $q=q_0^r$ for some divisor $r$ of $f$ with $r\geq 2$.  Let $G_0=G\cap\GL(d,q_0)$. Note that $c_p(\Z(\GL(d,q)))=0$ hence $c_p(G)=c_p(G_0)$. Since $q_0=p^{f/r}$ and $(d,f/r)<(d,f)$ in our lexicographic ordering, the inductive hypothesis yields $c_p(G_0)\le (\eps_{q_0}d-1)/(p-1)$. Since $q$ is a power of $q_0$, it follows from the definition of $\eps$ that $\eps_{q_0}\leq \eps_q$ and the result follows.

\subsection{$G\in\C_6$}\label{C6}
In this case, $d=r^m$ for some prime $r$ with $r\mid (q-1)$, and $G$ normalises an absolutely irreducible $r$-subgroup $R$ where $R/\Z(R)$ is elementary abelian of rank~$2m$. By~\cite[Proposition~4.6.5]{KL},  the normaliser of $R$ in $\SL(d,q)$ is 
 $$\Z(\SL(d,q))\circ\left(R\cdot\Sp(2m,r)\right).$$
 Since $r\mid(q-1)$, we have $r\neq p$ and  thus $c_p(\Z(\SL(d,q)))=c_p(R)=0$. It follows by Lemma~\ref{L1} that 
$$c_p(G)\le\log_p|\Sp(2m,r)|_p=\log_p\prod_{i=1}^m(r^{2i}-1)_p.$$ 
It thus suffices to show that
\begin{equation}\label{E6.5}
  \log_p\prod_{i=1}^m(r^{2i}-1)_p\le\frac{\eps_qr^m-1}{p-1}.
\end{equation}
Let $\Delta=\prod_{i=1}^m(r^{2i}-1)$. Suppose first that $p=2$ and thus $r\geq 3$. Note that
\begin{equation*}
  \log_2 \Delta_2=\log_2\prod_{i=1}^m (r^{2i}-1)_2<\log_2\prod_{i=1}^m r^{2i}=(m^2+m)\log_2 r.
\end{equation*}
If $(m^2+m)\log_2 r\le r^m-1$, then, clearly,~\eqref{E6.5} holds. We may thus assume that $(m^2+m)\log_2 r> r^m-1$ and it is not hard to see that this implies that $(m,r)$ is one of $(1,3)$, $(1,5)$ or $(2,3)$. If $(m,r)=(1,5)$, then $\log_2\Delta_2=3$ and~\eqref{E6.5} follows by noting that $\eps_q\geq 1$. Finally, if $r=3$ then $q$ must be an even power of $2$ and thus $\eps_q=4/3$ and again~\eqref{E6.5} can be verified directly for $m\in\{1,2\}$.

From now on, we assume that $p\geq 3$. Let $\ell$ be the order of $r^2$ modulo~$p$, that is, the smallest integer $\ell\geq 1$ for which $(r^2)^{\ell}\equiv 1\pmod p$. The key observation which follows from~\cite[Lemma~2.2(i)]{AHN} is that
\begin{equation*}
  (r^{2i}-1)_p=\begin{cases}1&\textup{ if $\ell\nmid i$,}\\
              (r^{2\ell}-1)_p\left(\frac{i}{\ell}\right)_p
            &\textup{ if $\ell\mid i$.}\end{cases}
\end{equation*}
Let $(r^{2\ell}-1)_p=p^e$ and note that $e\geq 1$. Hence
\begin{equation*}
  \Delta_p=\prod_{i=1}^m(r^{2i}-1)_p
  =\prod_{j=1}^{\lfloor m/\ell\rfloor}(r^{2j\ell}-1)_p
  =\prod_{j=1}^{\lfloor m/\ell\rfloor}(r^{2\ell}-1)_pj_p
  =p^{\lfloor m/\ell\rfloor e}\left(\left\lfloor \frac{m}{\ell}\right\rfloor!\right)_p.
\end{equation*}
Thus, by Lemma~\ref{lem:n!},
$\log_p\Delta_p\le\lfloor m/\ell\rfloor e+(\lfloor m/\ell\rfloor-1)/(p-1)$.
To prove~\eqref{E6.5}, it thus suffices to prove
\begin{equation*}
  \lfloor m/\ell\rfloor e+\frac{\lfloor m/\ell\rfloor-1}{p-1}
  \le\frac{\eps_qr^m-1}{p-1}.
\end{equation*}
For this, it is sufficient to show that 
\begin{equation}\label{yayap}
  \lfloor m/\ell\rfloor ep\le \eps_qr^m.
\end{equation}

Suppose first that $p=2^n+1$ is a Fermat prime and $r=2$. In this case $\ell=n$, $e=1$ and $\eps_q=\frac{p}{p-1}$. Hence~\eqref{yayap} becomes $\lfloor m/n \rfloor \le 2^{m-n}$. Writing $\alpha=m/n$, this inequality becomes $\lfloor \alpha \rfloor \le 2^{n(\alpha-1)}$, which holds for all values of $\alpha$ since $n\geq 1$.

We now assume that $p\ge 3$ is not a Fermat prime or $r\geq 3$, and hence that $\eps_q=1$. Since $p^e\mid (r^{2\ell}-1)$, we see that $p^e\mid(r^{\ell}\pm1)$ and hence $p^e\le r^{\ell}+1$. Suppose that equality holds. Since $p\geq 3$, we have $r=2$ hence $p^e-1$ is a power of $2$ and thus so is $p-1$. In other words, $p$ is a Fermat prime, contradicting our assumption. We may thus assume that $p^e\le r^{\ell}$ and hence $e\le\ell \log_p r$. Using~\eqref{yayap}, it suffices to prove $mp\log_p r\le r^m$.

We first consider the subcase when $p\le r^{m/2}$. Under this hypothesis, it suffices to prove $m\log_p r\le r^{m/2}$ and, since $p\geq 3$, even  $m\log_3r\le r^{m/2}$ is sufficient. It is not hard to show that this always holds.

We now assume that $p\geq r^{m/2}+1$ and hence $\ell\geqslant m/2$. If $\ell=m/2$, then $p=r^{m/2}+1$ and hence $r=2$ and $p$ is a Fermat prime, contrary to our hypothesis. Thus $\ell >m/2$. If $\ell>m$, then~\eqref{yayap} clearly holds. We may thus assume that  $m/2<\ell \leqslant m$ and hence $\Delta_p=(r^{2\ell}-1)_p$. 

Suppose that $p^2\mid (r^{2\ell}-1)$. Since $p\geq 3$, this implies that $p^2\mid r^\ell\pm1$ and thus $p^2\le r^\ell+1$. If $p^2=r^\ell+1$, then $r=2$ and $p=3$. We may thus assume that $p^2\le r^\ell$ and hence $p\le r^{\ell/2}\le r^{m/2}$, contrary to our hypothesis. Therefore $p^2\nmid(r^{2\ell}-1)$, and it follows that $\Delta_p=p$. In particular,~\eqref{E6.5} holds since $p\leq r^m$. This concludes the proof of this case.

\subsection{$G\in\C_8$}\label{C8}
In this case $G$ has a normal subgroup $N$ such that $N=\Omega^\varepsilon(d,q)$ for $d$ even or $dq$ odd, $\Sp(d,q)'$ for $d$ even, or $\SU(d,\sqrt{q})$ for $q$ a square.  Moreover, $G$ is contained in $\GO^\varepsilon(d,q)$, $\GSp(d,q)$ or $\GU(d,\sqrt{q})Z$, where $\GO^\varepsilon(d,q)$ and $\GSp(d,q)$ denote the groups of all similarities of the quadratic or alternating form respectively, while $\GU(d,\sqrt{q})$ denotes the group of all isometries of the hermitian form.  By excluding previous cases, we may also assume that  $N/(N\cap Z)$ is nonabelian and simple~\cite[\S VI.1--2]{N}. Thus $c_p(N)=0$ and hence $c_p(G)=c_p(G/N)$.

Now, $|\GU(d,\sqrt{q})Z:\SU(d,\sqrt{q})|=q-1$ which is coprime to $p$ hence $c_p(G)=0$ when $N=\SU(d,\sqrt{q})$. 

Similarly, $|\GSp(d,q):\Sp(d,q)|=q-1$, while $\Sp(d,q)'=\Sp(d,q)$ unless $(d,q)=(4,2)$. Thus, if $N=\Sp(d,q)'$, then $c_p(G)=0$ unless $(d,q)=(4,2)$, in which case $c_p(G)=1$. In both cases,~\eqref{E1} holds.

Finally, $|\GO^\varepsilon(d,q):\Omega^\varepsilon(d,q)|=2(q-1)\gcd(2,d,q-1)$ which is coprime to $p$ unless $p=2$. Thus, if $N=\Omega^\varepsilon(d,q)$, then $c_p(G)=0$ unless $p=2$, in which case $c_p(G)=1$. Again,~\eqref{E1} holds in both cases.

\subsection{$G\in\C_9$}\label{C9}
In this case, $G$ has a normal series $G\rhdeq N\rhd Z\rhdeq 1$ where $G/Z$ is almost simple with socle $N/Z$ and, moreover, $N$ is absolutely irreducible. Let $T=N/Z$. Note that $c_p(Z)=c_p(T)=0$ and thus $c_p(G)=c_p(G/N)$. Note also that $G/N$ is isomorphic to a subgroup of $\Out(T)$. It follows by Lemma~\ref{L1} that $c_p(G/N)\leq\log_p|\Out(T)|_p$.

If $|\Out(T)|_p\leq 2$, then $c_p(G)=c_p(G/N)\le c_p(\Out(T))\le1$ with
equality if and only if $p=2$. Certainly~\eqref{E1} is satisfied if
$p>2$, and it is satisfied when $p=2$ provided $1\le\eps_2d-1$. This is true
as $d\ge2$ and $\eps_2=4/3$. We
may thus assume that $|\Out(T)|_p\geq 3$. This already rules out the
case when $T$ is a sporadic group or an alternating group $\alt(n)$, with
$n\neq 6$. In view of the exceptional isomorphism
$\alt(6)\cong\PSL(2,9)$, we will therefore assume that $T$ is a
nonabelian simple group of Lie type. We rule out the Tits
group~${}^2\mathrm{F}_4(2)'$ as we view it as a sporadic group.

Suppose that $T$ is defined over a field $F'$ of characteristic $p'$ and order $(p')^{f'}$. Let $q'=|F'|=(p')^{f'}$ if $T$ is an untwisted group of Lie type, and $(q')^k=|F'|=(p')^{f'}$ if $T$ is twisted with respect to a graph symmetry of order $k$.

It is well known that $|\Out(T)|=\delta f'\gamma$ where $\delta$ and $\gamma$ are the number of ``diagonal'' and ``graph'' outer automorphisms, respectively (see~\cite[p.\;(xv)]{C} and~\cite[p.(xvi) Table~5]{C}). It follows that $c_p(G)\le\log_p(\delta f'\gamma)_p$. We now split into two cases, according to whether or not $p=p'$.

\subsubsection{$p=p'$}\label{Pablo}
By~\cite[Table~5]{C}, $\delta$ is coprime to $p$ and thus $\delta_p=1$.  We first suppose that $p\leq 3$ and $\gamma_p=1$. Recall that the field automorphisms yield a  cyclic subgroup of $\Out(T)$ of order $f'$, while a Sylow $p$-subgroup of $\GL(d,p^f)$ has exponent $p^{\lceil\log_p d \rceil}$(see~\cite[\S16.5]{HB3}, for example). It follows that $\log_p f'_p\le\lceil\log_p d\rceil$ hence
\begin{equation}\label{NewNew}
  c_p(G)\le\log_p f'_p \le\lceil\log_p d\rceil.
\end{equation}
When $p=2$, we have $\eps_q\ge1$ and
$\lceil\log_2 d\rceil\le d-1$ always holds.
When $p=3$, we have $\eps_q=3/2$  and $\lceil\log_3 d\rceil\le (\eps_qd-1)/2$
always holds. Thus~\eqref{E1} is true in this case.

We may now assume that either $p\geq 5$ or $\gamma_p\neq 1$. In particular, $T$ is neither a Suzuki group nor a Ree group (these have $p\le3$ and $\gamma=1$). By~\cite[p.\,(xv)]{C}, $\Out(T)$ has the form $(O_D\rtimes O_F)\rtimes O_G$ where $O_D,O_F,O_G$ denote groups of diagonal, field, and graph outer automorphisms, respectively. Conjugation induces on $N/Z\cong T$ a homomorphism $\overline{\phantom{n}}\colon G\to\Out(T)$, with kernel containing~$N$. We must bound $c_p(G)=c_p(\overline{G})$. Since $\delta_p=1$,
$O_D$ is a $p'$-group. Write $|\overline{G}\cap(O_D\rtimes O_F)|_p=p^\ell$. Then $c_p(G)\leqslant\ell+\log_p\gamma_p$ where $\log\gamma_p\le 1$ for $p\le3$, and $\log\gamma_p=0$ otherwise. We digress from bounding $c_p(G)$ (for three paragraphs) to show that $G$ contains an element of order $p^{\ell+1}$. This is trivially true if $\ell=0$ so assume that $\ell\geqslant 1$.

Choose $H\leqslant G$ such that $N\leqslant H$, $\overline{H}\leqslant O_D\rtimes O_F$, and $|H:N|=p^\ell$. Since $O_F$ is cyclic and $|O_D|_p=\delta_p=1$, Sylow's Theorem implies that  $H/Z$ is unique up to isomorphism. Thus we may assume that $H=\langle N,\varphi\rangle$, where the automorphism $\widetilde{\varphi}\in\Aut(T)$  induced by $\varphi$ on $T=N/Z$ is a standard field automorphism of order $p^\ell$.



Suppose first that $T$ is an untwisted group of Lie type.  Since $\widetilde{\varphi}$ is a standard field automorphism there is a root system  $\Phi$ for $T$ such that $T$ is generated by the set of all root elements  $x_r(\lambda)$ for $r\in \Phi$ and $\lambda\in F'$, and there is an automorphism $\psi$  of the field $F'$ of order $p^\ell$ such that $\widetilde{\varphi}\in\Aut(T)$ maps each $x_r(\lambda)$ to $x_r(\lambda^\psi)$ (see \cite{carter}).
Let $(F')^\psi$ be the fixed subfield of $\psi$ and let $\mathrm{Tr}\colon F'\to(F')^\psi$ be the (surjective) trace map $\mathrm{Tr}(\lambda)=\sum_{i=0}^{p^\ell-1} \lambda^{\psi^i}$. Calculating in $\Aut(T)$, with $T$ identified with $\Inn(T)$, we have
\begin{align*}
(\widetilde{\varphi} x_r(\lambda))^{p^\ell} & = (x_r(\lambda))^{\psi^{p^\ell-1}}(x_r(\lambda))^{\psi^{p^\ell-2}}\cdots (x_r(\lambda))^\psi x_r(\lambda)\\
                            &= x_r(\lambda^{\psi^{p^\ell-1}})x_r(\lambda^{\psi^{p^\ell-2}})\cdots x_r(\lambda^\psi)x_r(\lambda)\\
                           &=x_r(\lambda^{\psi^{p^\ell-1}}+\lambda^{\psi^{p^\ell-2}}+\cdots + \lambda^\psi+\lambda)=x_r(\mathrm{Tr}(\lambda)).
\end{align*}
Choosing $\lambda\in F'$ such that $\mathrm{Tr}(\lambda)\neq 0$ yields an element $\widetilde{\varphi}x_r(\lambda)$ of order $p^{\ell+1}$. Thus $H$, and hence $G$, has an element of
order $p^{\ell+1}$, as desired.

Suppose now that $T$ is a twisted group of Lie type  arising from an untwisted group $L$ with root system $\Phi$. Since $T$ is twisted, $\gamma=1$ hence $p\geq 5$ and  all roots in a fundamental system for $\Phi$ have the same length. Moreover, there is a graph automorphism $\rho$ of order $k$ arising from a symmetry of the Dynkin diagram of $L$ and a field automorphism $\sigma$ of order $k$ such that $T$ is the centraliser in $L$ of the automorphism $\rho\sigma$. By \cite[Proposition 13.6.3]{carter}, if $k=2$ and $T\neq \PSU(3,q')$, then there is a root $r$ with image $\overline{r}$ under the symmetry of the Dynkin diagram such that, for all $\lambda\in F'$, the element $x_S(\lambda):=x_r(\lambda)x_{\overline{r}}(\lambda^\sigma)$ lies in $T$. Similarly, if $k=3$, then there is a root $r$ with images $\overline{r}$ and $\overline{\overline{r}}$ such that, for all $\lambda\in F'$, the element $x_S(\lambda)=x_r(\lambda)x_{\overline{r}}(\lambda^\sigma)x_{\overline{\overline{r}}}(\lambda^{\sigma^2})$ lies in $T$. In both cases, a calculation similar to the earlier one shows that $(\widetilde{\varphi} x_S(\lambda))^{p^\ell}=x_S(\mathrm{Tr}(\lambda))$ and hence, by choosing $\lambda$ appropriately, we ensure that $\widetilde{\varphi} x_S(\lambda)$ has order $p^{\ell+1}$. Finally, if $T=\PSU(3,q')$, then, for a simple root $r$, we have $r+\overline{r}\in\Phi$ and hence $T$ contains elements $x_{r+\overline{r}}(\lambda)$ for all $\lambda$ in the index $2$ subfield of $F'$ fixed by the field automorphism of order 2. Since $p$ is odd, such a subfield contains elements with nonzero trace and we again find an element $\widetilde{\varphi} x_{r+\overline{r}}(\lambda)$ of order $p^{\ell+1}$.

We have shown that, in all cases, $G$ contains an element of order $p^{\ell+1}$.  Recall that a Sylow $p$-subgroup of $\GL(d,p^f)$ has exponent $p^m$  where $p^{m-1}<d\le p^m$. Thus  $\ell+1 \leqslant m$ and 
$$
  \ell\leqslant m-1 \leqslant\frac{p^{m-1}-1}{p-1}<\frac{d-1}{p-1}.
$$
In particular,~\eqref{E1} holds if  $c_p(G)=\ell$. We may thus assume that  $c_p(G)>\ell$ which implies that $\gamma_p\neq 1$ and $p\leq 3$. By~\cite[Table~5]{C}, $\gamma$ divides~$6$ hence  $\log_p\gamma_p=1$. It follows that
$$
  c_p(G)=\ell+1\leqslant m =\lceil\log_p d\rceil
$$
but, as we saw earlier in the sentences following~\eqref{NewNew}, this implies~\eqref{E1} when $p\leq 3$.

\subsubsection{$p\neq p'$.}
In this case, we have an  absolutely irreducible cross-characteristic representation $N\to\GL(d,q)$. This gives rise to a projective representation $T\to\PGL(d,q)$ and Landazuri and Seitz~\cite[Theorem]{LS} give lower bounds on $d$ with respect to $q'$.  Furthermore, possibilities for quasisimple groups $N$ and small dimensions $d$ are listed in~\cites{HM1,HM2}.

We first assume that $d\le5$. Suppose that $T\cong\PSL(2,q')$. By~\cite[Table~2]{HM1}, we have $d\in\{q',q'\pm1,(q'\pm1)/2\}$. Since $d\le 5$, this implies that $q'\le11$ and $q'\ne8$ and, as $|F'|=q'=(p')^{f'}$, we see that $f'\le2$. Since $|\Out(T)|$ divides $2f'$ and $|\Out(T)|_p\ge3$, it follows that $p=f'=2$. As $p'\ne p$, this implies that $q'=9$ and thus $d\ge(9-1)/2=4$ hence~\eqref{E1} holds.  Suppose now that $T$ is a group of Lie type other than $\PSL(2,q')$. By~\cite[Table~2]{HM2}, the possible choices for $T$ with $d\le5$ are $\PSL(3,4)$ and $\PSU(4,2)$ with $|\Out(T)|$ being $12$ and $2$, respectively. As $p\neq p'$ and $|\Out(T)|_p\geq 3$, we have $|\Out(T)|_p=p=3$ and $c_p(G)\leq 1$ hence~\eqref{E1} holds. We henceforth assume that $d\ge6$.

Suppose first that $\delta\geq 5$. This implies that $T=\PSL(n,q')$ or $\PSU(n,q')$ and $n\geq 4$. It follows by~\cite[Theorem]{LS} that 
\begin{equation}\label{yoyoya}
d\geq \frac{q'((q')^4-1)}{q'+1}=q'(q'-1)((q')^2+1).
\end{equation}
(Note that the exceptions for $\PSL(n,q')$ and $\PSU(n,q')$ in~\cite[Theorem]{LS} do not arise because $n\geq 4$.) Since $T=\PSL(n,q')$ or $\PSU(n,q')$, it follows by~\cite[Table~5]{C} that $\delta=\gcd(n+1,q'\pm 1)\le q'+1$ and thus
\begin{equation}\label{yoyoy}
d\geq q'(q'-1)((q')^2+1)\geq (\delta-1)(\delta-2)\delta\geq 12\delta.
\end{equation}
Similarly,~\eqref{yoyoya} implies $d\geq (q')^3$. As $(p')^{f'}=(q')^k$ for some $k\le2$, we have
\begin{equation}\label{yoyoyi}
f'=k\log_{p'} q'\leq2\log_{p'} q'\leq2\log_2 q'\leq2\log_2 d^{1/3}\leq 2(d-1)/3.
\end{equation}
Combining~\eqref{yoyoy} and~\eqref{yoyoyi} gives $\delta+f'\leq d-1$ and thus
\begin{align*}
  c_p(G)\le\log_p\gamma_p+\log_p\delta_p+\log_p f'_p&\le\frac{(p-1)\log_p\gamma_p}{p-1}+\frac{\delta-1}{p-1}+\frac{f'-1}{p-1}\\
  &\le\frac{(p-1)\log_p\gamma_p+d-3}{p-1}.
\end{align*}
If $p\geq 5$, then $\gamma_p=1$ and thus~\eqref{E1} holds. If $p\leq 3$, then $\log_p\gamma_p\leq 1$ and $p-1\leq 2$ and again~\eqref{E1} holds.

We may thus assume that $\delta\leq 4$. We will show that $f'\leq \log_{p'}(d+1)^2$. First, suppose that $k=1$. It follows by~\cite[Theorem]{LS} that $d\geq(q'-1)/2$. (As $d\geq 6$, we can assume that $q'>13$, which rules out the exceptional cases in~\cite[Theorem]{LS}.) This implies that 
$$ f'=\log_{p'}q'\le\log_{p'}(2d+1)< \log_{p'}(d+1)^2.$$
Next, if $k=2$, then~\cite[Theorem]{LS} implies that $d\geq q'-1$. This implies that
$$ f'=2\log_{p'}q'\le2\log_{p'}(d+1)=\log_{p'}(d+1)^2.$$
Finally, if $k=3$, then $d\geq (q')^3$ by~\cite[Theorem]{LS} and
$$ f'=3\log_{p'}q'\le3\log_{p'}d^{1/3}<\log_{p'}(d+1)^2.$$
This completes our proof that $f'\leq \log_{p'}(d+1)^2$.

Suppose first that $p\geq 5$.  By~\cite[Table~5]{C}, we have $\gamma_p=1$ and $\delta\leq 4$ implies that $\delta_p=1$. As $d\geq 6$, we have $f'\leq \log_{p'}(d+1)^2\leq \log_{2}(d+1)^2\leq d$. It follows that
$$
  c_p(G)\le\log_p f'_p\le\frac{f'-1}{p-1}\le\frac{d-1}{p-1}\le\frac{\eps_q d-1}{p-1},
$$
as desired. For $p=3$, we have
$$
 c_p(G)\le\log_3\gamma_3+\log_3\delta_3+\log_3 f'\le 1+1+\log_3\log_2(d+1)^2.
$$
It is not hard to see that $2+\log_3\log_2(d+1)^2\leq \frac{(3/2)d-1}{2}=\frac{\eps_3 d-1}{3-1}$ when $d\geq 6$. This completes the case $p=3$. Finally, suppose that $p=2$ and thus $p'\geq 3$. It follows that
$$
 c_p(G)\le\log_2\gamma_2+\log_2\delta_2+\log_2 f'\le 1+2+\log_2\log_3(d+1)^2.
$$
Again, it is not hard to see that $3+\log_2\log_3(d+1)^2\leq d-1$ when $d\geq 6$, establishing the case $p=2$. This completes the induction and thus the proof. \hfill\qed

\begin{corollary}\label{Cor2}
Let $V=(\F_q)^d$ be the natural module for $G\le\GL(d,q)$ where $q=p^f$. If $V$ has
a composition series with $r$ simple factors, and $\eps_q$ is defined
by~\eqref{E1}, then
\[
  c_p(G/O_p(G))\le \frac{\eps_q d-r}{p-1}.
\]
\end{corollary}

\begin{proof}
Fix a composition series $V>V_1>\dots>V_r=\{0\}$ for $V$ and consider the
homomorphism $\phi\colon G\to\prod_{i=1}^r\GL(W_i)$ where $W_i:=V_{i-1}/V_i$
for $1\le i\le r$. 
Let $G_i$ be the subgroup of $\GL(W_i)$ induced by $G$. Then
$G_i$ acts irreducibly on $W_i$. Hence the largest normal $p$-subgroup
$O_p(G_i)$ of $G_i$ is trivial. (Note that $[O_p(G_i),W_i]$ is $G_i$-invariant
and $[O_p(G_i),W_i]<W_i$, so $[O_p(G_i),W_i]=\{0\}$.) It follows that $\ker(\phi)=O_p(G)$.

We have $d=d_1+\cdots+d_r$ where $d_i=\dim(W_i)$. Applying
Theorem~\ref{T1} gives
\[
  c_p(G/O_p(G))=c_p(G/\ker(\phi))=c_p(\textup{im}(\phi))\le\sum_{i=1}^r c_p(G_i)
  \le\sum_{i=1}^r\frac{\eps_q d_i-1}{p-1}=\frac{\eps_q d-r}{p-1}.\qedhere
\]
\end{proof}

\bigskip

\noindent\textsc{Acknowledgements.}
We would like to thank Pablo Spiga for some of the ideas in Section~\ref{Pablo}.
We thank the referee for a number of helpful suggestions including
the above corollary.

\end{document}